\documentclass[a4paper, 11pt]{amsart}

\usepackage{amsmath, amssymb, amsthm, mathtools}
\usepackage[colorlinks=true, linkcolor=blue]{hyperref}

\usepackage[margin=1in]{geometry}

\usepackage{listings}

\newtheorem{theorem}{Theorem}[section]
\newtheorem{lemma}[theorem]{Lemma}
\newtheorem{proposition}[theorem]{Proposition}
\newtheorem{corollary}[theorem]{Corollary}

\theoremstyle{definition}
\newtheorem{definition}[theorem]{Definition}
\newtheorem{example}[theorem]{Example}
\newtheorem{remark}[theorem]{Remark}
\newtheorem{question}[theorem]{Question}

\newcommand{\NN}{\mathbb{N}}
\newcommand{\ZZ}{\mathbb{Z}}
\newcommand{\QQ}{\mathbb{Q}}
\newcommand{\RR}{\mathbb{R}}
\newcommand{\FF}{\mathbb{F}}
\newcommand{\PP}{\mathbb{P}}

\newcommand{\dimgr}{\dim_{\mathrm{gr}}}
\newcommand{\Spec}{\operatorname{Spec}}
\newcommand{\height}{\operatorname{ht}}
\newcommand{\trdeg}{\operatorname{tr.deg}}
\newcommand{\gkdim}{\operatorname{GKdim}}
\newcommand{\Frac}{\operatorname{Frac}}
\newcommand{\rank}{\operatorname{rank}}
\renewcommand{\div}{\operatorname{div}}

\title[On Krull's Dimension Theorem for Certain Graded Rings and Its Applications]{On Krull's Dimension Theorem for Certain Graded Rings and Its Applications}

\author{Rirai Ikeda}
\address{Department of Mathematics, School of Science, Science Tokyo, 2-12-1 Ookayama, Meguro-ku, Tokyo 152-8551, Japan}
\email{ikeda.r.cfbf@m.isct.ac.jp}

\begin{document}

\begin{abstract}
    This paper explores the dimension theory of non-Noetherian graded rings by introducing the class of Hilbert--Serre rings. Motivated by Krull's dimension theorem and Smoke's dimension theorem, we establish the fundamental inequalities
    \begin{align*}
        \dimgr(R) \leq \dim(R) \leq \gkdim_k(R) \leq d(R)
    \end{align*}
    for any Hilbert--Serre ring $R$, where $d(R)$ is the pole order of its Poincar\'e series at $t=1$. We then prove that all these dimensions, including the transcendence degree in the domain case, coincide for monomial algebras. As an application, let $S$ be a finitely generated homogeneous subalgebra of a polynomial ring over a field $k$, and let $\mathrm{in}_{<}(S)$ be its initial algebra with respect to a monomial order $<$. Even though $\mathrm{in}_{<}(S)$ need not be Noetherian, we prove that taking initial algebras preserves all the dimensions considered in this paper; in particular,
    \begin{align*}
        \dim(S)=\dim(\mathrm{in}_{<}(S)).
    \end{align*}
    Finally, we provide explicit examples demonstrating that these inequalities can be strict in general, even for Hilbert--Serre domains.
\end{abstract}

\maketitle

\tableofcontents

\section{Introduction}

Throughout this paper, all rings are commutative with identity, and all ring homomorphisms preserve identity.

Dimension theory for Noetherian rings has been one of the central topics in commutative algebra. One of the crowning achievements in this area is Krull's dimension theorem \cite[Theorem 13.4]{Mat86}, which provides a fundamental understanding of the Krull dimension. In the graded setting, Smoke's dimension theorem states that for a Noetherian $\NN$-graded ring $R$, the Krull dimension coincides with the order of the pole at $t=1$ of its Poincar\'e series, denoted by $d(R)$ (Definition \ref{def:hilbertserre}). 

\begin{theorem}[Smoke's dimension theorem ({\cite[Theorem 5.5]{Smo72}})]\label{th:smoke}
    Let $R = \bigoplus_{n \geq 0} R_n$ be a Noetherian $\NN$-graded ring with $R_0$ a field. Then we have: 
    \begin{equation*}
        \dim(R) = d(R) = s(R),
    \end{equation*}
    where $\dim(R)$ is the Krull dimension of $R$, $d(R)$ is the order of the pole at $t=1$ of the Poincar\'e series of $R$, i.e., 
    \begin{align*}
        d(R) = \min \left\{ d \in \NN \;\middle|\; \lim_{t \to 1-0} (1-t)^d P_R(t) < \infty \right\},
    \end{align*}
    and $s(R)$ is the minimal number of homogeneous elements of $R$ such that $R$ is a finitely generated module over the subalgebra generated by those elements over $R_0$.
\end{theorem}

This theorem can be applied to the study of dimensions of initial algebras. Let
\begin{align*}
    S=k[f_1,\ldots,f_m]\subset k[x_1,\ldots,x_n]
\end{align*}
be a homogeneous subalgebra of a polynomial ring over a field $k$ generated by homogeneous elements $f_1, \ldots, f_m$, and let $<$ be a monomial order. The initial algebra $\mathrm{in}_{<}(S)$, recalled in Definition \ref{def:inisagbi}, is the $k$-subalgebra generated by the initial monomials of elements of $S$, i.e., $\mathrm{in}_{<}(S) = k[\mathrm{in}_{<}(g) \mid g \in S]$. Since $\mathrm{in}_{<}(S)$ is a monomial algebra (i.e., generated by only monomials), its dimension is often easier to compute than that of $S$. In particular, the theorem of Arnold--Gilmer (\cite{AG76}, \cite[Theorem 21.4]{Gil84}) can be used to compute its Krull dimension from the associated monoid, as will be illustrated in Section \ref{sec:initial}. Moreover, taking initial algebras preserves the Poincar\'e series:
\begin{align*}
    P_S(t)=P_{\mathrm{in}_{<}(S)}(t),
\end{align*}
as claimed in \cite[Proposition 2.4]{CHV96}. Since $S$ is finitely generated, Smoke's dimension theorem applies to $S$. If $\mathrm{in}_{<}(S)$ is also Noetherian, then Smoke's dimension theorem gives
\begin{align*}
    \dim(S)
    &=d(S) \\
    &=d(\mathrm{in}_{<}(S)) \\
    &=\dim(\mathrm{in}_{<}(S)).
\end{align*}
Thus, in this case, passing to the initial algebra preserves the Krull dimension, and one may compute the dimension of $S$ through the monomial algebra $\mathrm{in}_{<}(S)$. However, $\mathrm{in}_{<}(S)$ need not be Noetherian in general (Example \ref{ex:nonnoeth1} and Example \ref{ex:nonnoeth2}). This motivates a non-Noetherian extension of Smoke's dimension theorem.

For this purpose, we introduce the class of Hilbert--Serre rings in Definition \ref{def:hilbertserre}. This class is designed to include graded rings whose Poincar\'e series has a finite pole order at $t=1$, even without assuming Noetherianity; for such a ring $R$, this pole order is denoted by $d(R)$ and is called the Hilbert--Serre dimension of $R$. By the Hilbert--Serre theorem, every Noetherian $\NN$-graded ring over a field is a Hilbert--Serre ring. On the other hand, this class also contains many non-Noetherian graded rings, including monomial subalgebras of polynomial rings. Thus, Hilbert--Serre rings provide a natural setting in which one can ask how much of Smoke's dimension theorem remains valid after removing the Noetherian hypothesis.

For a Hilbert--Serre ring $R$, the invariant $d(R)$ is defined without assuming Noetherianity. It is therefore natural to ask whether $d(R)$ still measures the Krull dimension, as in Smoke's dimension theorem. In the non-Noetherian setting, one can also compare several other dimension invariants which are closely related to the Krull dimension in Noetherian situations. In this paper, we consider the following five invariants:
\begin{itemize}
    \item $\dimgr(R)$: the graded Krull dimension of $R$,
    \item $\dim(R)$: the Krull dimension of $R$,
    \item $\trdeg_k(R)$: the transcendence degree of $R$ over $k$, defined only when $R$ is a domain,
    \item $\gkdim_k(R)$: the Gelfand--Kirillov dimension of $R$ over $k$,
    \item $d(R)$: the order of the pole at $t=1$ of the Poincar\'e series of $R$.
\end{itemize}
Our first main result shows that the equalities in Smoke's dimension theorem do not hold for arbitrary Hilbert--Serre rings, but that they are replaced by the following chain of inequalities.
\begin{theorem}[Theorem {\ref{th:main}}]\label{th:intro-main}
    Let $R = \bigoplus_{n \geq 0} R_n$ be an $\NN$-graded ring with $R_0 = k$ a field. If $R$ is a Hilbert--Serre ring, then we have:
    \begin{align*}
        \dimgr(R) \leq  \dim(R) \leq \gkdim_k(R) \leq d(R).
    \end{align*}
    Moreover, if $R$ is a domain, then we have:
    \begin{align*}
        \dimgr(R) \leq \dim(R) \leq \trdeg_k(R) = \gkdim_k(R) \leq d(R).
    \end{align*}
\end{theorem}
The appearance of the Gelfand--Kirillov dimension is natural from this point of view. The transcendence degree is available only for domains, while $\gkdim_k(R)$ is defined for arbitrary $k$-algebras. Moreover, in the domain case these two invariants agree, as stated in Theorem \ref{th:intro-main}. Thus, $\gkdim_k(R)$ plays the role of $\trdeg_k(R)$ in the general, possibly non-domain, case. The examples in Section \ref{sec:examples} show that the inequalities in Theorem \ref{th:intro-main} can be strict in general, even for Hilbert--Serre domains.

Although the inequalities in Theorem \ref{th:intro-main} can be strict in general, they become equalities for monomial algebras. This is the second main result of this paper and can be viewed as a non-Noetherian version of Smoke's dimension theorem for monomial algebras.
\begin{theorem}[Theorem {\ref{th:monomial}}]\label{th:intro-monomial}
Let $S$ be a monomial algebra over a field $k$. In other words, $S$ is a $k$-subalgebra of a polynomial ring $R = k[x_1, \ldots, x_n]$ generated by monomials, not necessarily Noetherian. Regarding $S$ as an $\NN$-graded ring with the standard grading inherited from $R$ as mentioned in Definition \ref{def:monomial}, we have:
\begin{align*}
\dimgr(S) = \dim(S) = \trdeg_k(S) = \gkdim_k(S) = d(S).
\end{align*}
\end{theorem}
Thus, while Smoke's dimension theorem does not extend to all Hilbert--Serre rings as an equality, it does extend in this strong form to monomial algebras. The proof uses the theorem of Arnold--Gilmer to compute the Krull dimension of a monomial algebra through the group associated with its defining monoid.

We now return to the application to initial algebras discussed above. Since $\mathrm{in}_{<}(S)$ is a monomial algebra, Theorem \ref{th:intro-monomial} applies to it even when $\mathrm{in}_{<}(S)$ is not Noetherian. Combining this with the preservation of the Poincar\'e series under taking initial algebras, we obtain the following consequence.
\begin{corollary}[Corollary {\ref{cor:initialdim}}]\label{cor:intro-initial}
    Let $S$ be a finitely generated homogeneous subalgebra of a polynomial ring $R = k[x_1, \ldots, x_n]$ over a field $k$ and $<$ be a monomial order on $R$. Then, all dimensions of $S$ and $\mathrm{in}_<(S)$ are the same, i.e., we have:
    \begingroup
    \setlength{\arraycolsep}{0.3em}
    \begin{equation*}
        \begin{array}{*{10}{c}}
            &
            \dimgr(S)
            & = &
            \dim(S)
            & = &
            \trdeg_k(S)
            & = &
            \gkdim_k(S)
            & = &
            d(S)
            \\
            =
            &
            \dimgr(\mathrm{in}_{<}(S))
            & = &
            \dim(\mathrm{in}_{<}(S))
            & = &
            \trdeg_k(\mathrm{in}_{<}(S))
            & = &
            \gkdim_k(\mathrm{in}_{<}(S))
            & = &
            d(\mathrm{in}_{<}(S)).
        \end{array}
    \end{equation*}
    \endgroup
\end{corollary}
In particular, taking initial algebras preserves the Krull dimension:
\begin{align*}
    \dim(S)=\dim(\mathrm{in}_{<}(S)).
\end{align*}
Thus, even when $\mathrm{in}_{<}(S)$ is non-Noetherian, one can compute the dimension of $S$ by passing to the monomial algebra $\mathrm{in}_{<}(S)$. In Example \ref{ex:initial-nonnoeth-jacobian}, we illustrate this method by a finitely generated homogeneous algebra whose initial algebra is non-Noetherian; its dimension is computed from the associated monoid using the theorem of Arnold--Gilmer, even in a characteristic where the usual Jacobian computation does not detect the dimension.

In Section \ref{sec:examples}, we show that the inequalities in Theorem \ref{th:intro-main} are strict in general. More precisely, we construct Hilbert--Serre domains for which all inequalities
\begin{align*}
    \dimgr(R) \leq \dim(R) \leq \trdeg_k(R) = \gkdim_k(R) \leq d(R)
\end{align*}
are strict. Thus, Smoke's dimension theorem does not extend to arbitrary Hilbert--Serre domains as a statement of equalities. These examples lead to the natural problem stated in Question \ref{qu:open}: what conditions on integers $a,b,c,d,e \in \NN$ are necessary and sufficient for the existence of a Hilbert--Serre domain $R$ such that
\begin{align*}
    \dimgr(R)=a,\ \dim(R)=b,\ \trdeg_k(R)=c,\ \gkdim_k(R)=d,\ d(R)=e\ ?
\end{align*}
The necessary conditions include $a \leq b \leq c=d \leq e$ by Theorem \ref{th:intro-main} and $\lfloor b/2 \rfloor \leq a$ by Proposition \ref{prop:dimupper}, but a complete characterization remains open.

The paper is organized as follows. In Section \ref{sec:preliminaries}, we recall basic facts on graded rings, Poincar\'e series, graded Krull dimension, and homogeneous prime ideals. In Section \ref{sec:hilbertserre}, we introduce Hilbert--Serre rings and establish their basic properties, including the finiteness of the Krull dimension. Section \ref{sec:dimineq} is devoted to the proof of the main dimension inequalities for Hilbert--Serre rings. In Section \ref{sec:monomial}, we prove a non-Noetherian version of Smoke's dimension theorem for monomial algebras. In Section \ref{sec:initial}, we apply this result to initial algebras and prove the dimension preservation result stated in Corollary \ref{cor:initialdim}. We also give an example illustrating how to compute the dimension of a finitely generated homogeneous algebra by passing to a non-Noetherian initial algebra. In Section \ref{sec:examples}, we give examples and counterexamples showing that the hypotheses and inequalities in our results are sharp, and we end with Question \ref{qu:open}.

\subsection*{Acknowledgments}

The author would like to thank the members of his laboratory, Kazuma Shimomoto, Tatsuki Yamaguchi, Ryo Ishizuka, Kazuki Hayashi, Yuto Yamada, Taiga Ozaki, and Tomoya Kobayashi, for careful proofreading and valuable suggestions on wording, style, and the title. The author is deeply grateful to Kanau Shimada for his valuable comments on the construction of Example \ref{ex:3}, to Shigeru Kuroda for his insightful suggestions regarding Example \ref{ex:nonnoeth2}, and to Samuel Griffiths for providing an interesting counterexample which is now incorporated as Example \ref{ex:countergr}. Furthermore, the author thanks Kaiji Kondo for his valuable comments on Definition \ref{def:hilbertserre} and Takehiro Hosaka for his helpful comments that improved the readability of the proofs.

\section{Preliminaries}\label{sec:preliminaries}

We define the Poincar\'e series for $\mathbb{Z}$-graded modules in order to state Lemma \ref{lem:nonzerodiv}. After proving this lemma, we will only consider the Poincar\'e series for an $\NN$-graded ring $R$.

\begin{definition}\label{def:poincare}
    Let $R = \bigoplus_{n \geq 0} R_n$ be a graded ring and $M = \bigoplus_{n \in \ZZ} M_n$ be a graded $R$-module. We assume that $R_0$ is a field $k$ and $M_n = 0$ for all sufficiently small $n \in \ZZ$ and that $\dim_k M_n < \infty$ for all $n \in \ZZ$. Then the Poincar\'e series of $M$ is defined as
    \begin{equation*}
        P_M(t) = \sum_{n \in \ZZ} (\dim_k M_n) t^n.
    \end{equation*}
\end{definition}

In this section, let $R = \bigoplus_{n \geq 0} R_n$ and $M = \bigoplus_{n \in \ZZ} M_n$ be a graded $R$-module as in Definition \ref{def:poincare} unless otherwise stated.

\begin{lemma}\label{lem:nonzerodiv}
    Let $x \in R_h$ be a homogeneous element of degree $h \in \NN$. If $x$ is a non-zero-divisor on $M$, then we have
    \begin{equation*}
        P_{M/xM}(t) = (1 - t^h) P_M(t).
    \end{equation*}
\end{lemma}
\begin{proof}
    Since $x$ is a non-zero-divisor on $M$, we have the following short exact sequence of graded $R$-modules:
    \begin{equation*}
        0 \to M(-h) \xrightarrow{\cdot x} M \to M/xM \to 0,
    \end{equation*}
    where $M(-h)$ is the graded module defined by $M(-h)_n = M_{n-h}$ for all $n \in \ZZ$. Taking the Poincar\'e series of this exact sequence, and using the fact that $P_{M(-h)}(t) = t^h P_M(t)$, we get
    \begin{equation*}
        P_M(t) = t^h P_M(t) + P_{M/xM}(t).
    \end{equation*}
    Rearranging this equation gives us the desired result.
\end{proof}

Let $\dim(R)$ be the Krull dimension of $R$ and $\dimgr(R)$ the \emph{graded Krull dimension} of $R$ defined as the supremum of the lengths of chains of homogeneous prime ideals in $R$, i.e.,
\begin{equation*}
    \dimgr(R) = \sup \left\{ r \in \NN \;\middle|\; \begin{array}{l} \text{there exists a chain of homogeneous } \\ \text{prime ideals } \mathfrak{p}_0 \subsetneq \mathfrak{p}_1 \subsetneq \cdots \subsetneq \mathfrak{p}_r \end{array} \right\}.
\end{equation*}
By definition, we have $\dimgr(R) \leq \dim(R)$. When $R$ is a Noetherian $\NN$-graded ring with finite Krull dimension, we have $\dimgr(R) = \dim(R)$; see \cite[Lemma 1.37]{Blu18}. However, in general, we do not have the equality if we do not assume the Noetherian condition (but assume $\NN$-gradedness). 

For a given prime ideal $\mathfrak{p}$ of $R$, we define the homogenization of $\mathfrak{p}$ as the ideal $\mathfrak{p}^*$ generated by the homogeneous elements in $\mathfrak{p}$, i.e.,
\begin{equation*}
    \mathfrak{p}^* = ( x \in R \mid x \text{ is homogeneous and } x \in \mathfrak{p} ).
\end{equation*}
It is well known that $\mathfrak{p}^*$ is a homogeneous prime ideal of $R$ and $\mathfrak{p}^* \subseteq \mathfrak{p}$ holds. Note that $\mathfrak{p}^* = \mathfrak{p}$ if and only if $\mathfrak{p}$ is a homogeneous ideal. 

We shall use the following lemma.

\begin{lemma}\label{lem:laurant}
    Let $R = \bigoplus_{n \in \ZZ} R_n$ be a $\ZZ$-graded ring. The following are equivalent:
    \begin{itemize}
        \item[(1)] $R_0$ is a field and there exists an invertible homogeneous element $T \in R$ such that $R = R_0[T, T^{-1}]$ as a graded ring.
        \item[(2)] All non-zero homogeneous elements of $R$ are invertible.
    \end{itemize}
\end{lemma}
\begin{proof}
    (1) $\Rightarrow$ (2): Clear.

    (2) $\Rightarrow$ (1): It is immediate that $R_0$ is a field. If $R_n = 0$ for all $n > 0$, then $R = R_0$, and the assertion is clear. Let $n>0$ be the smallest integer such that $R_n \neq 0$. Take a non-zero homogeneous element $T \in R_n$. Then $T$ is invertible by assumption.

    We claim that if $R_m \neq 0$, then $n$ divides $m$. It is enough to consider the case $m>0$. Indeed, if $m<0$ and $R_m \neq 0$, then for any non-zero homogeneous element $a \in R_m$, its inverse $a^{-1}$ is a non-zero homogeneous element of degree $-m>0$.

    Now assume $m>0$ and $R_m \neq 0$. Write $m=cn+r$ with $c \in \NN$ and $0 \leq r < n$. Take a non-zero homogeneous element $a \in R_m$. Since $T$ is invertible, $aT^{-c}$ is a non-zero homogeneous element of degree $r$. Hence $R_r \neq 0$. By the minimality of $n$, we must have $r=0$. Thus $n$ divides $m$.

    Therefore, if $R_m \neq 0$, then $m=kn$ for some $k \in \ZZ$. For any $a \in R_{kn}$, we have $aT^{-k} \in R_0$, so $a \in R_0T^k$. Hence $R_{kn}=R_0T^k$ for all $k \in \ZZ$. Thus, $R = R_0[T,T^{-1}]$ as a graded ring.
\end{proof}

Lemma \ref{lem:roberts} is found in \cite[Lemma 1]{MR74}. However, their proof relies on the Noetherian condition. Here we give a proof without the Noetherian condition by using Lemma \ref{lem:laurant} and the localization technique.

\begin{lemma}\label{lem:roberts}
    Let $R = \bigoplus_{n \in \NN} R_n$ be an $\NN$-graded ring and $\mathfrak{p} \in \Spec(R)$ be a non-homogeneous prime ideal of $R$. Then we have $\height(\mathfrak{p}/\mathfrak{p}^*) = 1$.
\end{lemma}
\begin{proof}
    Let $S$ be the set of all homogeneous elements of $R$ which are not in $\mathfrak{p}$, i.e., 
    \begin{equation*}
        S = \{ x \in R \mid x \text{ is homogeneous and } x \notin \mathfrak{p} \}.
    \end{equation*}
    Then $S$ is a multiplicatively closed set. We can see $S^{-1}\mathfrak{p}$ and $S^{-1}\mathfrak{p}^*$ are prime ideals of the $\ZZ$-graded ring $S^{-1}R$ and $S^{-1}\mathfrak{p}^*$ is a homogeneous prime ideal of $S^{-1}R$. Moreover, we can see that $S^{-1}R/S^{-1}\mathfrak{p}^*$ satisfies the condition (2) in Lemma \ref{lem:laurant}. Hence, we have that $S^{-1}R/S^{-1}\mathfrak{p}^*$ has Krull dimension 1 (since it is not a field). Then we have
    \begin{equation*}
        \height(\mathfrak{p}/\mathfrak{p}^*) = \height(S^{-1}\mathfrak{p}/S^{-1}\mathfrak{p}^*) = 1.
    \end{equation*}
\end{proof}

Our last preliminary in this section is the following proposition. We have $\dimgr(R) \leq \dim(R)$ and the equality does not hold in general. However, can the difference between these two dimensions get arbitrarily large? The answer is no. 

\begin{proposition}\label{prop:dimupper}
    Let $R = \bigoplus_{n \in \NN} R_n$ be an $\NN$-graded ring with $R_0$ a field. Then, if there is a chain of prime ideals of $R$ of length $r$, then there is a chain of homogeneous prime ideals of $R$ of length at least $\left\lfloor \frac{r}{2} \right\rfloor$, where $\left\lfloor \cdot \right\rfloor$ denotes the floor function. In particular, we have 
    \begin{equation*}
        \left\lfloor \frac{\dim(R)}{2} \right\rfloor \leq \dimgr(R) \leq \dim(R), 
    \end{equation*}
    including the case $\dim(R) = \infty$. 
\end{proposition}
\begin{proof}
    The case $r = 0, 1$ is clear since $R_+ = \bigoplus_{n > 0} R_n$ is the maximal homogeneous ideal of $R$. We assume $r > 1$. If we have a chain of prime ideals $\mathfrak{p}_0 \subsetneq \mathfrak{p}_1 \subsetneq \mathfrak{p}_2$ of length 2, we have $\mathfrak{p}_0^* \subseteq \mathfrak{p}_1^* \subseteq \mathfrak{p}_2^*$ as homogeneous prime ideals. However, it is impossible that $\mathfrak{p}_0^* = \mathfrak{p}_1^* = \mathfrak{p}_2^*$ holds. If so, we have
    \begin{equation*}
        \mathfrak{p}_2^* \subseteq \mathfrak{p}_0 \subsetneq \mathfrak{p}_1 \subsetneq \mathfrak{p}_2.
    \end{equation*}
    This means $\height(\mathfrak{p}_2/\mathfrak{p}_2^*) \geq 2$, which contradicts Lemma \ref{lem:roberts}. Hence, we have at least one strict inclusion in the chain $\mathfrak{p}_0^* \subseteq \mathfrak{p}_1^* \subseteq \mathfrak{p}_2^*$. Then, if we take a chain of prime ideals $\mathfrak{p}_0 \subsetneq \mathfrak{p}_1 \subsetneq \cdots \subsetneq \mathfrak{p}_r$ of length $r$, we have $\mathfrak{p}_0^* \subseteq \mathfrak{p}_1^* \subseteq \cdots \subseteq \mathfrak{p}_r^*$ as a chain of homogeneous prime ideals. These inclusions can be strict or not. However, two inclusions next to each other cannot be both equal. Hence, we have at least $\left\lfloor \frac{r}{2} \right\rfloor$ strict inclusions in that chain. This means we have a chain of homogeneous prime ideals of length at least $\left\lfloor \frac{r}{2} \right\rfloor$.
\end{proof}

Note that using the above proposition, we can conclude that $\dim(R) < \infty$ if and only if $\dimgr(R) < \infty$ for an $\NN$-graded ring $R$ with $R_0$ a field.

\section{Hilbert--Serre Rings}\label{sec:hilbertserre}

In general, the Poincar\'e series of graded rings do not necessarily have a well-defined pole order at $t = 1$. If $R$ is Noetherian, the Poincar\'e series of $R$ has a well-defined pole order at $t=1$ by the Hilbert--Serre theorem; see \cite[Theorem 11.1]{AM69}. However, the converse does not hold, i.e., there are non-Noetherian graded rings whose Poincar\'e series has a well-defined pole order at $t = 1$. Therefore, it is worth defining the following class of graded rings.

\begin{definition}\label{def:hilbertserre}
    Let $R = \bigoplus_{n \geq 0} R_n$ be an $\NN$-graded ring with $R_0 = k$ a field. We say $R$ is a \emph{Hilbert--Serre ring} if $R$ satisfies the following conditions:
    \begin{itemize}
        \item[(1)] $\dim_k R_n < \infty$ for all $n \in \NN$.
        \item[(2)] The Poincar\'e series $P_R(t)$ of $R$ has radius of convergence 1 or $\infty$. 
        \item[(3)] $D := \left\{ d \in \NN \mid \lim_{t \to 1-0} (1-t)^d P_R(t) < \infty \right\}$ is not empty.
    \end{itemize}
    We define $d(R)$ as the minimum of $D$, i.e., 
    \begin{equation*}
        d(R) = \min \left\{ d \in \NN \;\middle|\; \lim_{t \to 1-0} (1-t)^d P_R(t) < \infty \right\}.
    \end{equation*}
    We call $d(R)$ the \emph{Hilbert--Serre dimension} of $R$.
\end{definition}

The name of this class of graded rings comes from the Hilbert--Serre theorem, which states that if $R$ is a Noetherian $\NN$-graded ring, then $P_R(t)$ is a rational function of the form $f(t)/\prod_{i=1}^d (1-t)^{\lambda_i}$, where $f(t) \in \ZZ[t]$. Thus, every Noetherian $\NN$-graded ring is a Hilbert--Serre ring. Here are some examples and non-examples.

\begin{example}\label{ex:hilbertserre}
    Let $R = \bigoplus_{n \geq 0} R_n$ be an $\NN$-graded ring with $R_0 = k$ a field. 
    \begin{itemize}
        \item[(1)] If $R$ is a Noetherian $\NN$-graded ring, then $R$ is a Hilbert--Serre ring and $d(R) = \dim(R)$ by Smoke's dimension theorem; see \cite[Theorem 5.5]{Smo72}. Note that $\dimgr(R) = \dim(R)$ holds if $R$ is a Noetherian $\NN$-graded ring. 
        \item[(2)] $R = k[x, xy, xy^2, xy^3, \ldots] \subset k[x, y]$ is a non-Noetherian but Hilbert--Serre ring. Indeed, we have 
            \begin{equation*}
                P_R(t) = 1 + \sum_{n=1}^{\infty} nt^n = 1 + t + 2t^2 + 3t^3 + \cdots = \frac{t^2-t+1}{(1-t)^2}.
            \end{equation*}
            In this case, we have $\dimgr(R) = \dim(R) = d(R) = 2$. We provide more details on this example in Example \ref{ex:1} in the later section.
        \item[(3)] $R = k[xy, xy^2, xy^3, \ldots] \subset k[x, y]$ is a non-Noetherian but Hilbert--Serre ring. Indeed, we have 
            \begin{equation*}
                P_R(t) = 1 + \sum_{n=1}^{\infty} \left\lfloor \frac{n}{2} \right\rfloor t^n = 1 + t^2 + t^3 + 2t^4 + 2t^5 + \cdots = \frac{t^4+t^3-t^2+1}{(1-t^2)^2}.
            \end{equation*}
            In this case, we have $\dimgr(R) = \dim(R) = d(R) = 2$. We provide more details on this example in Example \ref{ex:2} in the later section.
        \item[(4)] $R = k[x_1, x_2, x_3, \ldots]$, a polynomial ring in infinitely many variables, is a non-Noetherian and non-Hilbert--Serre ring because $\dim_k(R_1) = \infty$. In this case, we cannot even define the Poincar\'e series of $R$.
        \item[(5)] $R = k[x_1, x_2^2, x_3^3, \ldots]$ is a non-Noetherian and non-Hilbert--Serre ring. A Hilbert--Serre ring actually has finite Krull dimension (Corollary \ref{cor:finitedimension}). So we have that this $R$ is not a Hilbert--Serre ring. However, unlike the previous example, we can define the Poincar\'e series of $R$ as 
            \begin{equation*}
                P_R(t) = 1 + t + 2t^2 + 3t^3 + 5t^4 + 7t^5 + 11t^6 + \cdots.
            \end{equation*}
            In other words, the Poincar\'e series is the generating function of the partition function $p(n)$. In this case, we have $\dimgr(R) = \dim(R) = \infty$. For precise details on this example, see Example \ref{ex:6} in a later section.
    \end{itemize}
\end{example}

By straightforward observations, we have the following properties of Hilbert--Serre rings.

\begin{proposition}\label{prop:hilbertserre}
    Let $R = \bigoplus_{n \geq 0} R_n$ be an $\NN$-graded ring with $R_0 = k$ a field and $R$ a Hilbert--Serre ring.
    \begin{itemize}
        \item[(1)] If $I$ is a homogeneous ideal of $R$, then $R/I$ is also a Hilbert--Serre ring and we have $d(R/I) \leq d(R)$.
        \item[(2)] If $x \in R_h$ is a homogeneous non-zero-divisor and non-unit, i.e., an $R$-regular element of $R$, then $R/xR$ is also a Hilbert--Serre ring and we have $d(R/xR) = d(R) - 1$.
        \item[(3)] If $S = \bigoplus_{n \geq 0} S_n$ is an $\NN$-graded ring with $S_0 = R_0 = k$ a field and $\phi: S \to R$ is a homogeneous injective ring homomorphism, then $S$ is also a Hilbert--Serre ring and we have $d(S) \leq d(R)$.
    \end{itemize}
\end{proposition}
\begin{proof}
    In general, let
    \begin{equation*}
        \begin{split}
            f(t) &= a_0 + a_1 t + a_2 t^2 + \cdots, \\ 
            g(t) &= b_0 + b_1 t + b_2 t^2 + \cdots
        \end{split}
    \end{equation*}
    be two power series with non-negative coefficients and $a_n \leq b_n$ for all $n \in \NN$. If there exists $d \in \NN$ such that $\lim_{t \to 1-0} (1-t)^d g(t) < \infty$, then we have:
    \begin{equation*}
        \lim_{t \to 1-0} (1-t)^d f(t) \leq \lim_{t \to 1-0} (1-t)^d g(t) < \infty.
    \end{equation*}
    This means that (1) and (3) hold. (2) is immediate from Lemma \ref{lem:nonzerodiv}. 
\end{proof}

By the above proposition, we find that there are more Hilbert--Serre rings than Noetherian $\NN$-graded rings. For example, a subring $S$ of a polynomial ring $R = k[x_1, \ldots, x_n]$ is a Hilbert--Serre domain if the natural inclusion $S \hookrightarrow R$ is a homogeneous injective ring homomorphism even if $S$ is not Noetherian. 

The following lemma is a first non-trivial property of Hilbert--Serre rings. This is a weak version of our main Theorem \ref{th:main}. However, we need it to claim that Hilbert--Serre rings have finite Krull dimension (Corollary \ref{cor:finitedimension}). A similar argument appears in \cite[Proposition 11.10]{AM69}. 

\begin{lemma}\label{lem:finitegrdim}
    Let $R = \bigoplus_{n \geq 0} R_n$ be an $\NN$-graded ring with $R_0 = k$ a field. If $R$ is a Hilbert--Serre ring, then $\dimgr(R) \leq d(R)$.
\end{lemma}
\begin{proof}
    We show the claim by induction on $d = d(R)$. 
    
    In the case $d = 0$, we have to show $\dimgr(R) = 0$. Since the Poincar\'e series $P_R(t)$ does not have a pole at $t=1$, $R_n = 0$ for all sufficiently large $n \in \NN$. So we can find $m \in \NN$ such that $R_+^m = 0$, where $R_+ = \bigoplus_{n > 0} R_n$, the maximal homogeneous ideal of $R$. If we take an arbitrary prime ideal $\mathfrak{p}$ of $R$, we have $R_+^m \subseteq \mathfrak{p}$. This means that $\mathfrak{p} = R_+$. Hence, we have $\dimgr(R) = 0$.

    We assume $d > 0$ and that the claim holds for $d-1$. We assume that $\mathfrak{p}_0 \subsetneq \cdots \subsetneq \mathfrak{p}_r$ is a chain of homogeneous prime ideals of $R$ with length $r$. We can take a homogeneous element $x \in \mathfrak{p}_1 \setminus \mathfrak{p}_0$ since $\mathfrak{p}_0 \subsetneq \mathfrak{p}_1$. We put $R' = R/\mathfrak{p}_0$ and $x'$ the image of $x$ in $R'$. Then, $R'$ is also a Hilbert--Serre domain by Proposition \ref{prop:hilbertserre} (1) and $x'$ is a homogeneous non-zero-divisor and non-unit on $R'$. Moreover, there is at least a chain of homogeneous prime ideals of $R'/x'R'$ of length $r-1$ which is the image of the chain $\mathfrak{p}_1/\mathfrak{p}_0 \subsetneq \cdots \subsetneq \mathfrak{p}_r/\mathfrak{p}_0$ in $R'$. Hence, by the induction hypothesis and Proposition \ref{prop:hilbertserre} (2), we have
    \begin{equation*}
        r-1 \leq \dimgr(R'/x'R') \leq d(R'/x'R') = d(R') - 1 \leq d(R) - 1.
    \end{equation*}
    This means $r \leq d(R)$ and we have $\dimgr(R) \leq d(R)$.
\end{proof}

\begin{corollary}\label{cor:finitedimension}
    Let $R = \bigoplus_{n \geq 0} R_n$ be an $\NN$-graded ring with $R_0 = k$ a field. If $R$ is a Hilbert--Serre ring, then $\dim(R) < \infty$.
\end{corollary}
\begin{proof}
    If not, we can find for any $r \in \NN$ a chain of prime ideals of $R$ of length $r$. By Proposition \ref{prop:dimupper}, we have $\left\lfloor \frac{r}{2} \right\rfloor \leq \dimgr(R)$. Since $r$ can be arbitrarily large, we have $\dimgr(R) = \infty$, which contradicts Lemma \ref{lem:finitegrdim}.
\end{proof}

\section{Dimension Inequalities for Hilbert--Serre Rings}\label{sec:dimineq}

In this section, first we show that for a Hilbert--Serre domain, $\dimgr(R) \leq \dim(R) \leq \trdeg_k(R) \leq d(R)$ holds, where $R = \bigoplus_{n \geq 0} R_n$ is an $\NN$-graded ring with $R_0 = k$ a field and $\trdeg_k(R)$ is the transcendence degree of $R$ over $k$. Then we show that for a Hilbert--Serre ring (not necessarily a domain), we have $\dimgr(R) \leq \dim(R) \leq d(R)$. Second, we review the definition of the Gelfand-Kirillov dimension (GKdim). We then examine relations between GKdim and the other dimensions. Finally, we derive the main Theorem \ref{th:main} by combining these results.

\begin{lemma}\label{lem:dimtrdeg}
    Let $R$ be a domain containing a field $k$ as a subring. Then we have $\dim(R) \leq \trdeg_k(R)$.
\end{lemma}
\begin{proof}
    Note that if $R$ is a finitely generated algebra over $k$, then the equality holds; see \cite[Theorem 5.6]{Mat86}.
    
    We take a chain of prime ideals $\mathfrak{p}_0 \subsetneq \cdots \subsetneq \mathfrak{p}_r$ of $R$ with length $r$. We take $x_i \in \mathfrak{p}_i \setminus \mathfrak{p}_{i-1}$ for each $i = 1, \ldots, r$. We put $S = k[x_1, \ldots, x_r]$. $\mathfrak{p}_0 \cap S \subsetneq \cdots \subsetneq \mathfrak{p}_r \cap S$ is a chain of prime ideals of $S$ with length $r$. Hence, we have: 
    \begin{equation*}
        r \leq \dim(S) = \trdeg_k(S) \leq \trdeg_k(R).
    \end{equation*}
    This completes the proof.
\end{proof}

In the above lemma, the equality does not hold in general if $R$ is not a finitely generated algebra over $k$ even if $R$ is Noetherian. For example, let $R = k(X)$ be the field of rational functions in one variable $X$ over a field $k$. Then we have $\dim(R) = 0$ but $\trdeg_k(R) = 1$.

Using the above lemma, we can show the following lemma.

\begin{lemma}\label{lem:trdegd}
    Let $R = \bigoplus_{n \geq 0} R_n$ be an $\NN$-graded ring with $R_0 = k$ a field. If $R$ is a Hilbert--Serre domain, then $\trdeg_k(R) \leq d(R)$.
\end{lemma}
\begin{proof}
    Let $\frac{r_1}{s_1}, \ldots, \frac{r_t}{s_t} \in \Frac(R)$ be algebraically independent elements over $k$, where $\Frac(R)$ is the field of fractions of $R$ and $r_i, s_i \in R$. Let $H$ be the set of all homogeneous elements of $R$ which appear in the homogeneous decomposition of $r_i$ and $s_i$ for some $i = 1, \ldots, t$. Then $k(h \mid h \in H)$ clearly contains all $\frac{r_i}{s_i}$. We can find $h_1, \ldots, h_t \in H$ such that these $h_1, \ldots, h_t \in R$ are algebraically independent over $k$. 

    We put $T = k[h_1, \ldots, h_t]$. Then $T$ is a polynomial ring in $t$ variables over $k$ and the natural inclusion $T \hookrightarrow R$ is a homogeneous injective ring homomorphism. Hence, by Proposition \ref{prop:hilbertserre} (3), we have $t = d(T) \leq d(R)$. This means $\trdeg_k(R) \leq d(R)$.
\end{proof}

Note that for a polynomial ring $R = k[x_1, \ldots, x_n]$ in $n$ variables over a field $k$, we have $\dim(R) = \trdeg_k(R) = d(R) = n$.

At this point, we have $\dim(R) \leq \trdeg_k(R) \leq d(R)$ for a Hilbert--Serre domain $R$. Using this fact, we can improve the inequality $\dimgr(R) \leq d(R)$ into $\dim(R) \leq d(R)$ for a Hilbert--Serre ring $R$ as follows.

\begin{lemma}\label{lem:dimd}
    Let $R = \bigoplus_{n \geq 0} R_n$ be an $\NN$-graded ring with $R_0 = k$ a field. If $R$ is a Hilbert--Serre ring, then $\dim(R) \leq d(R)$.
\end{lemma}
\begin{proof}
    By Corollary \ref{cor:finitedimension}, we have $\dim(R) < \infty$. We can take a chain of prime ideals $\mathfrak{p}_0 \subsetneq \cdots \subsetneq \mathfrak{p}_r$ of $R$ with length $r = \dim(R)$. Now $\mathfrak{p}_0$ is a homogeneous prime ideal. If not, its homogenization $\mathfrak{p}_0^*$ is a homogeneous prime ideal of $R$ and $\mathfrak{p}_0^* \subsetneq \mathfrak{p}_0$ holds. This means we have the chain of prime ideals:
    \begin{equation*}
        \mathfrak{p}_0^* \subsetneq \mathfrak{p}_0 \subsetneq \mathfrak{p}_1 \subsetneq \cdots \subsetneq \mathfrak{p}_r.
    \end{equation*}
    This contradicts the definition of $r = \dim(R)$. Hence, $\mathfrak{p}_0$ is a homogeneous prime ideal. Then we have:
    \begin{equation*}
        \dim(R) = \dim(R/\mathfrak{p}_0) \leq d(R/\mathfrak{p}_0) \leq d(R),
    \end{equation*}
    by Proposition \ref{prop:hilbertserre} (1) and Lemma \ref{lem:trdegd}. This completes the proof.
\end{proof}

Now, we define the Gelfand-Kirillov dimension (GKdim) and claim some basic known properties of it.

\begin{definition}\label{def:gkdim}
    Let $k$ be a field and $R$ be a $k$-algebra. For a finite-dimensional $k$-subspace $V$ of $R$, we put $V^n$ as the finite-dimensional subspace of $R$ spanned by all monomials of $n$ elements in $V$ over $k$. Then the Gelfand-Kirillov dimension of $R$ over $k$ is defined as
    \begin{equation*}
        \gkdim_k(R) = \sup_V \limsup_{n \to \infty} \frac{\log(d_V(n))}{\log(n)},
    \end{equation*}
    where $d_V(n) = \dim_k(k + V + \cdots + V^n)$ and the supremum is taken over all finite-dimensional $k$-subspaces $V$ of $R$ and $\log$ is the logarithm with base $e$.
\end{definition}

It is obvious that the equality
\begin{equation*}
    \gkdim_k(R) = \sup\left\{\gkdim_k(S) \mid S \subset R \text{ is a finitely generated } k\text{-subalgebra of } R\right\}
\end{equation*}
holds.

We record some basic properties of the Gelfand-Kirillov dimension without proofs. 

\begin{remark}\label{rem:gkdim}
    Let $k$ be a field. Then we have the following properties of the Gelfand-Kirillov dimension.
    \begin{itemize}
        \item[(1)] If $B$ is a subalgebra or a homomorphic image of $A$, then $\gkdim_k(B) \leq \gkdim_k(A)$. (\cite[Lemma 3.1]{KL00})
        \item[(2)] If $A$ is finitely generated as a $k$-algebra, then $\gkdim_k(A) = \dim(A)$. (\cite[Theorem 4.5]{KL00})
    \end{itemize}    
\end{remark}

Now, we can see that $\trdeg_k(R) = \gkdim_k(R)$ for a $k$-algebra $R$ that is a domain by using the following Lemma \ref{lem:trdeggkdim}. 

\begin{lemma}\label{lem:trdeggkdim}
    Let $R$ be a domain containing a field $k$ as a subring. Then we have:
    \begin{equation*}
        \trdeg_k(R) = \sup\left\{\trdeg_k(S) \mid S \subset R \text{ is a finitely generated } k\text{-subalgebra of } R\right\}. 
    \end{equation*}
    Therefore, we have $\trdeg_k(R) = \gkdim_k(R)$.
\end{lemma}
\begin{proof}
    $\geq$ is clear. For $\leq$, let $\frac{r_1}{s_1}, \ldots, \frac{r_n}{s_n} \in \Frac(R)$ be any finite algebraically independent elements over $k$. We put $S = k[r_i, s_i \mid i = 1, \ldots, n]$. Then $S$ is a finitely generated $k$-algebra and $\trdeg_k(S) \geq n$. Since the finite algebraically independent set was arbitrary, the desired inequality follows.
\end{proof}

At this point, we have 
\begin{equation*}
    \dimgr(R) \leq \dim(R) \leq \trdeg_k(R) = \gkdim_k(R) \leq d(R)
\end{equation*}
for a Hilbert--Serre domain $R$. Using this fact, we can show Theorem \ref{th:main} which is our main result. 

\begin{theorem}\label{th:main}
    Let $R = \bigoplus_{n \geq 0} R_n$ be an $\NN$-graded ring with $R_0 = k$ a field. If $R$ is a Hilbert--Serre ring, then we have the following first assertion:
    \begin{equation*}
        \dimgr(R) \leq \dim(R) \leq \gkdim_k(R) \leq d(R).
    \end{equation*}
    Moreover, if $R$ is a domain, then we have the following second assertion:
    \begin{equation*}
        \dimgr(R) \leq \dim(R) \leq \trdeg_k(R) = \gkdim_k(R) \leq d(R).
    \end{equation*}
\end{theorem}
\begin{proof}
    The second assertion is already shown in Lemma \ref{lem:dimtrdeg}, Lemma \ref{lem:trdeggkdim}, and Lemma \ref{lem:trdegd}. We have to show the first assertion. By Corollary \ref{cor:finitedimension} and the same argument as in the proof of Lemma \ref{lem:dimd}, we can take a homogeneous minimal prime ideal $\mathfrak{p}$ of $R$ such that $\dim(R/\mathfrak{p}) = \dim(R)$. Then we have
    \begin{equation*}
        \dim(R) = \dim(R/\mathfrak{p}) \leq \gkdim_k(R/\mathfrak{p}) \leq \gkdim_k(R),
    \end{equation*}
    by Remark \ref{rem:gkdim} (1) and the second assertion. Finally, we have to show that $\gkdim_k(R) \leq d(R)$. We take a finite-dimensional $k$-subspace $V$ of $R$. Since $V$ is finite-dimensional, we can find $m \in \NN$ such that $V \subseteq \bigoplus_{i=0}^m R_i$. Hence, we have
    \begin{equation*}
        d_V(n) = \dim_k(k + V + \cdots + V^n) \leq \dim_k\left(\bigoplus_{i=0}^{mn} R_i\right) = \sum_{i=0}^{mn} \dim_k(R_i),
    \end{equation*}
    for sufficiently large $n \in \NN$. We put $d = d(R)$. Since we assumed that $R$ is a Hilbert--Serre ring, we have $\lim_{t \to 1-0} (1-t)^d P_R(t) < \infty$. This means that there exists $C > 0$ and $\delta > 0$ such that for all $t \in (1-\delta, 1)$, we have $P_R(t) < \frac{C}{(1-t)^d}$. Hence, we have
    \begin{equation*}
        d_V(n)t^{mn} \leq \sum_{i=0}^{mn} \dim_k(R_i)t^{mn} \leq \sum_{i=0}^{mn} \dim_k(R_i)t^i \leq P_R(t) < \frac{C}{(1-t)^d}
    \end{equation*}
    for all $t \in (1-\delta, 1)$. Since $n$ is sufficiently large, we have $t = 1 - \frac{1}{n} \in (1-\delta, 1)$. Then, we can find $C' > 0$ such that
    \begin{equation*}
        d_V(n) < \frac{C'}{(1-t)^d} = C' n^d
    \end{equation*}
    for all sufficiently large $n \in \NN$. Hence, we can calculate as follows:
    \begin{equation*}
        \begin{split}
            \gkdim_k(R) &= \sup_V \limsup_{n \to \infty} \frac{\log(d_V(n))}{\log(n)} \\
            &\leq \sup_V \limsup_{n \to \infty} \frac{\log(C' n^d)}{\log(n)} \\ 
            &= d = d(R).
        \end{split}
    \end{equation*}
    Our proof is complete.
\end{proof}

We can now observe that the inequality 
\begin{equation*}
    \dimgr(R) \leq \dim(R) \leq \trdeg_k(R) \leq d(R)
\end{equation*}
for a Hilbert--Serre domain $R$ can be generalized to the inequality
\begin{equation*}
    \dimgr(R) \leq \dim(R) \leq \gkdim_k(R) \leq d(R)
\end{equation*}
for a Hilbert--Serre ring $R$ by using the language of the Gelfand-Kirillov dimension.

As already mentioned in the Introduction, all inequalities in our main theorem are equalities when $R$ is Noetherian; see \cite[Theorem 5.5]{Smo72}. However, in general, these equalities do not hold even if $R$ is a domain. We will see examples in the next section. On the other hand, all inequalities are equalities if $R$ is a monomial algebra over $k$ (Theorem \ref{th:monomial}). 

\section{Smoke's Dimension Theorem for Monomial Algebras}\label{sec:monomial}

In the previous section, we established the general inequalities
\begin{equation*}
    \dimgr(R) \leq \dim(R) \leq \trdeg_k(R) = \gkdim_k(R) \leq d(R)
\end{equation*}
for Hilbert--Serre domains $R$ over a field $k$. Our main goal is to demonstrate that, for monomial algebras, all the dimensions coincide.

\begin{definition}[Monomial algebra]\label{def:monomial}
    Let $R = k[x_1, \ldots, x_n]$ be a polynomial ring over a field $k$. A \emph{monomial algebra} over $k$ is a $k$-subalgebra of $R$ generated by a set of monomials. Equivalently, a monomial algebra over $k$ is a $k$-algebra of the form $k[M] = k[\{ x^a \mid a \in M \}]$, where $M$ is an additive submonoid of $\NN^n$ and $x^a = x_1^{a_1} \cdots x_n^{a_n}$ for $a = (a_1, \ldots, a_n) \in M$. This algebra is naturally isomorphic to the monoid ring of $M$ over $k$. Note that a monomial algebra is an $\NN$-graded ring with the standard grading inherited from $R$, where $\deg(x_i) = 1$ for each $i = 1, \ldots, n$. This $k[M]$ is a Hilbert--Serre ring by Proposition \ref{prop:hilbertserre} (3). 
\end{definition}

To prove Smoke's dimension theorem for monomial algebras, we need the following theorem. This theorem also enables us to compute the Krull dimension of a monomial algebra $k[M]$ easily; see Example \ref{ex:initial-nonnoeth-jacobian}. 

\begin{theorem}[{\cite{AG76}, \cite[Theorem 21.4]{Gil84}}]\label{th:gilmer}
    Let $R$ be a commutative ring with identity, $S$ a cancellative monoid, and $G$ its group of fractions (sometimes called the Grothendieck group). Then we have
    \begin{equation*}
        \dim(R[S]) = \dim(R[G])
    \end{equation*}
    where $R[S]$ is the monoid algebra of $S$ over $R$ and $R[G]$ is the group algebra of $G$ over $R$.
\end{theorem}

In particular, if $R = k$ is a field and $S \subset \NN^n$ is a monoid, then we have $\dim(k[S]) = \dim(k[G])$ where $G$ is the group of fractions of $S$. Since $G$ is a free abelian group, we can find a $\ZZ$-basis $e_1, \ldots, e_g$ of $G$. Then we have:
\begin{align*}
    \dim(k[S]) = \rank(G)
\end{align*}

Using these results, we can show that Smoke's dimension theorem holds for monomial algebras. 

\begin{theorem}\label{th:monomial}
    Let $S$ be a monomial algebra over a field $k$. In other words, $S$ is a $k$-subalgebra of a polynomial ring $R = k[x_1, \ldots, x_n]$ generated by monomials (not necessarily Noetherian). Regarding $S$ as an $\NN$-graded ring with the standard grading inherited from $R$ as mentioned in Definition \ref{def:monomial}, we have:
    \begin{equation*}
        \dimgr(S) = \dim(S) = \trdeg_k(S) = \gkdim_k(S) = d(S).
    \end{equation*}
\end{theorem}
\begin{proof}
    We already have $\dimgr(S) \leq \dim(S) \leq \trdeg_k(S) = \gkdim_k(S) \leq d(S)$ by Theorem \ref{th:main}. We have to show $\dim(S) = d(S)$ and $\dimgr(S) = \dim(S)$. Let $M$ be a monoid such that $S = k[M]$. Note that $M$ is a cancellative monoid since $M$ is a submonoid of $\NN^n$. We put $G$ to be the group of fractions of $M$. Then we have $\dim(S) = \dim(k[M]) = \dim(k[G])$ by Theorem \ref{th:gilmer}. 

    ($\dim(S) = d(S)$): We put $m = \trdeg_k(S)$. $G$ is a free abelian group, so we can find:
    \begin{equation*}
        k[G] \cong k[x_1^{\pm 1}, \ldots, x_g^{\pm 1}]
    \end{equation*}
    where $g$ is the rank of $G$. It is immediate that $k[M] \subset k[G] \subset \Frac(k[M])$. Hence, we have:
    \begin{equation*}
        g = \trdeg_k(k[G]) = \trdeg_k(\Frac(k[M])) = \trdeg_k(S) = m.
    \end{equation*}
    Then we find $g = m$. Now, it is enough to show that $d(S) \leq m$. If we can do it, then we have:
    \begin{equation*}
        \dim(S) = \dim(k[G]) = g = m = \trdeg_k(S).
    \end{equation*}
    Thus, it is enough to show $d(S) \leq m$. First, we decompose $S$ into the direct sum of $\NN$-graded rings as follows:
    \begin{equation*}
        S = \bigoplus_{n=0}^\infty S_n. 
    \end{equation*}
    Let $A(N) = \sum_{n=0}^{N} \dim_k(S_n)$ for $N \in \NN$. Note that
    \begin{equation*}
        \frac{P_S(t)}{1-t} = \sum_{N=0}^\infty A(N) t^N
    \end{equation*}
    holds. We consider $M, G$ as subsets of $\RR^n$. Let $V = G \otimes_{\ZZ} \RR$ be a real vector space generated by $G$. Since $G$ is a free abelian group of rank $m$, we take a $\ZZ$-basis $e_1, \ldots, e_m$ of $G$. Then $e_1,\ldots,e_m$ form an $\RR$-basis of $V=G\otimes_{\ZZ}\RR$.
    We put:
    \begin{equation*}
        P_N = \left\{ \alpha = (\alpha_1, \ldots, \alpha_n) \in \RR^n \ \middle| \ \alpha \in V, 0 \leq \forall \alpha_i, \sum_{i=1}^{n} \alpha_i \leq N \right\}. 
    \end{equation*}
    Clearly, the number of lattice points in $P_N$ gives an upper bound of $A(N)$ and $P_N = NP_1$. We take $\beta \in P_1$. Since $\beta \in V$, we can find $y_1, \ldots, y_m \in \RR$ such that $\beta = \sum_{i=1}^m y_i e_i$. Since $P_1$ is bounded, there exists a constant $K > 0$ independent of $y_i$ such that $|y_i| < K$ for all $i$. If we take $\alpha \in P_N$, we find
    \begin{equation*}
        \frac{1}{N} \alpha = \sum_{i=1}^m \frac{1}{N} z_i e_i \in P_1. 
    \end{equation*}
    Hence, we have $|z_j| \leq KN$. All $z_j$ are integers if and only if $\alpha \in G$. These $z_j$ are at most $2KN + 1$ integers. Hence, the number of lattice points in $P_N$ is at most $(2KN + 1)^m$. Therefore, we can calculate as follows:
    \begin{equation*}
        A(N) \leq (2KN + 1)^m \leq (2K+1)^m N^m.
    \end{equation*}
    In other words, we can take $C>0$ which does not depend on $N$ such that $A(N) \leq CN^m$ for all $N \in \NN$. Hence, we have
    \begin{equation*}
        \frac{P_S(t)}{1-t} = \sum_{N=0}^\infty A(N) t^N \leq C\sum_{N=0}^\infty N^m t^N.
    \end{equation*}
    It is well-known that the series on the right-hand side has a pole of order $m+1$ at $t=1$. Hence, we have $d(S) \leq m$. This completes the proof. 

    ($\dimgr(S) = \dim(S)$): We have to show that $\dim(S) \leq \dimgr(S)$. The case when $\dim(S) = 0$ is trivial. We assume $\dim(S) > 0$. We can take a non-trivial monomial $x^a = x_1^{a_1}\cdots x_n^{a_n}$ in $S$ such that $0 \neq a = (a_1, \ldots, a_n) \in M$. Let $d = \deg(a) = a_1 + \cdots + a_n$. The localization $S_{x^a}$ is a $\ZZ$-graded ring in a natural way. We put $A = (S_{x^a})_0$. Then $A$ is a monomial algebra over $k$ with the following monoid:
    \begin{equation*}
        M' = \{ b - l a \in G \mid b \in M, l \in \NN, \deg(b - l a) = 0 \}.
    \end{equation*}
    In other words, $A = k[M']$. Let $G'$ be the group of fractions of $M'$. The following mapping is a group morphism:
    \begin{equation*}
        \deg \colon G \to \ZZ \colon g = (g_1, \ldots, g_n) \mapsto \deg(g) = g_1 + \cdots + g_n.
    \end{equation*}
    $\deg$ is not a zero morphism since $\deg(a) = d > 0$. Hence, $\rank(\ker(\deg)) = \rank(G) - 1 = g - 1$. Clearly, we have $M' \subset \ker(\deg)$. Thus $G' \subset \ker(\deg)$. Moreover, we can find $d \cdot \ker(\deg) \subset G'$. Take any $g \in \ker(\deg)$. Since $g \in G$, we can find $m_1, m_2 \in M$ such that $g = m_1 - m_2$. We put $c = \deg(m_1) = \deg(m_2)$. Then we have:
    \begin{equation*}
        d \cdot g = d \cdot m_1 - d \cdot m_2 = (d m_1 - c a) - (d m_2 - c a) \in G'
    \end{equation*}
    because $d m_i - c a \in M' \quad (i = 1,2)$. Hence, we have $d \cdot \ker(\deg) \subset G' \subset \ker(\deg)$ and thus $\rank(G') = \rank(\ker(\deg)) = g - 1$. Then we have:
    \begin{equation*}
        \dim(A) = \dim(k[M']) = \dim(k[G']) = \rank(G') = g - 1 = \dim(S) - 1.
    \end{equation*}
    Thus, we can find a chain of prime ideals $0 = \mathfrak{p}_0 \subsetneq \cdots \subsetneq \mathfrak{p}_{g-1}$ of $A$ of length $g - 1$. \\
    On the other hand, we consider the $d$-th Veronese subalgebra $(S_{x^a})^{(d)}$. We have:
    \begin{equation*}
        (S_{x^a})^{(d)} = \bigoplus_{l \in \ZZ} (S_{x^a})_{ld} = \bigoplus_{l \in \ZZ} A \cdot (x^a)^l = A[x^a, x^{-a}]. 
    \end{equation*}
    We can easily check that $x^a$ is algebraically independent over $A$. Hence, $(S_{x^a})^{(d)}$ is a Laurent polynomial ring over $A$. Therefore, we can take a chain of homogeneous prime ideals $0 = \mathfrak{p}_0[x^a, x^{-a}] \subsetneq \cdots \subsetneq \mathfrak{p}_{g-1}[x^a, x^{-a}]$ of $(S_{x^a})^{(d)} = A[x^a, x^{-a}]$ of length $g - 1$. Moreover, for all monomials $y \in S_{x^a}$, $y^d \in (S_{x^a})^{(d)}$. Therefore, $A[x^a, x^{-a}] = (S_{x^a})^{(d)} \hookrightarrow S_{x^a}$ is an integral extension. Thus, by the going-up theorem, we can find a prime ideal $\mathfrak{q}_1$ of $S_{x^a}$ such that $\mathfrak{q}_1 \cap A[x^a, x^{-a}] = \mathfrak{p}_1[x^a, x^{-a}]$. Since $\mathfrak{p}_1[x^a, x^{-a}] \subset \mathfrak{q}_1$ and $\mathfrak{p}_1[x^a, x^{-a}]$ is homogeneous, $\mathfrak{p}_1[x^a, x^{-a}] \subset \mathfrak{q}_1^* \subset \mathfrak{q}_1$ and thus $\mathfrak{q}_1^* \cap A[x^a, x^{-a}] = \mathfrak{p}_1[x^a, x^{-a}]$. Hence, we can assume $\mathfrak{q}_1$ is a homogeneous prime ideal of $S_{x^a}$. By the same argument, we can find a chain of homogeneous prime ideals $\mathfrak{q}_0 \subsetneq \cdots \subsetneq \mathfrak{q}_{g-1}$ of $S_{x^a}$ of length $g - 1$. Hence, we can find a chain of prime ideals $\mathfrak{q}_0 \cap S \subsetneq \cdots \subsetneq \mathfrak{q}_{g-1} \cap S$ of $S$ of length $g - 1$. All these homogeneous primes do not contain $x^a$. Therefore, we can find a chain of homogeneous prime ideals $\mathfrak{q}_0 \cap S \subsetneq \cdots \subsetneq \mathfrak{q}_{g-1} \cap S \subsetneq \bigoplus_{n>0}S_n$ of $S$ of length $g$. Hence, we have $\dimgr(S) \geq g = \dim(S)$. This completes the proof.
\end{proof}

As mentioned in the Introduction, this result is applied to the initial algebras of subalgebras of polynomial rings in the next section.

\section{Applications to Initial Algebras}\label{sec:initial}

In this section, we apply the results of the previous sections to initial algebras. First, let us briefly recall the standard setup and notation for initial algebras. Let $R = k[x_1, \ldots, x_n]$ be a polynomial ring over a field $k$. We recall that a monomial order $<$ on $R$ is a total order on the set of monomials in $R$ that satisfies the following properties:
\begin{itemize}
    \item[(1)] $1 < x_i$ for all $i$,
    \item[(2)] If $u < v$, then $uw < vw$ for all monomials $u, v, w$.
\end{itemize}
In the standard Gr\"obner basis theory, the initial operation is defined for ideals of $R$ with respect to a fixed monomial order. However, in this paper, we consider the initial "algebra", i.e., we consider the initial algebra of a subalgebra of $R$ instead of an ideal. 

\begin{definition}\label{def:inisagbi}
    Let $R = k[x_1, \ldots, x_n]$ be a polynomial ring over a field $k$ and $<$ be a monomial order on $R$. Let $f_1, \ldots, f_m \in R$ be homogeneous polynomials, and let $S = k[f_1, \ldots, f_m]$ be a homogeneous subalgebra of $R$. We define the \emph{initial algebra} $\mathrm{in}_<(S)$ of $S$ with respect to $<$ as follows:
    \begin{equation*}
        \mathrm{in}_<(S) = k[\mathrm{in}_<(g) \mid g \in S].
    \end{equation*}
    If a subset $\mathcal{F} \subset S$ generates the initial algebra $\mathrm{in}_<(S)$, i.e., $\mathrm{in}_<(S) = k[\mathrm{in}_<(g) \mid g \in \mathcal{F}]$, then we say $\mathcal{F}$ is a \emph{SAGBI basis} of $S$ with respect to $<$.
\end{definition}

$\mathcal{F}$ can be infinite. For basic results on initial algebras and SAGBI bases, see \cite{RS90}. Note that the initial algebra $\mathrm{in}_<(S)$ is a monomial algebra over $k$ defined in Definition \ref{def:monomial}. In other words, $\mathrm{in}_<(S)$ is a $k$-subalgebra of $k[x_1, \ldots, x_n]$ generated by monomials. 

One can define the initial algebra for non-homogeneous subalgebras of $R$. However, if we consider the non-homogeneous case, the subalgebra $S$ does not admit a natural $\NN$-grading and thus we cannot define the Poincar\'e series of $S$. Hence, we only consider the homogeneous case in this paper.

Of course, by the Hilbert basis theorem, any initial ideal is finitely generated in Gr\"obner basis theory, but an initial algebra is not necessarily finitely generated even if $S$ is a regular ring. In other words, an initial algebra is not necessarily Noetherian. 

\begin{example}[{\cite[1.20]{RS90}}]\label{ex:nonnoeth1}
    Let $R = k[x, y]$ be a polynomial ring over a field $k$ and $S = k[x+y, xy, xy^2]$ be a subalgebra. Then, for any monomial order $<$ on $R$, $\mathrm{in}_<(S)$ is not Noetherian. In this case, $S$ is not a regular ring but is a hypersurface ring because $S \cong k[X, Y, Z]/(X^2-XYZ+Z^3)$. 
\end{example}

Here is an example where $S$ is a regular ring but $\mathrm{in}_<(S)$ is not Noetherian.

\begin{example}\label{ex:nonnoeth2}
    Let $R = k[x, y, z]$ be a polynomial ring over a field $k$ and $S = k[x+y+z, xy, xy^2]$ be a subalgebra, and we define the monomial order $<$ as follows:
    \begin{equation*}
        x^{a_1}y^{a_2}z^{a_3} > x^{b_1}y^{b_2}z^{b_3} \Longleftrightarrow 
        \begin{cases}
            a_1 + a_2 > b_1 + b_2 \\
            \text{or } a_1 + a_2 = b_1 + b_2 \land (a_1, a_2, a_3) >_{\text{lex}} (b_1, b_2, b_3) \\
        \end{cases}
    \end{equation*}
    where $>_{\text{lex}}$ is the lexicographic order on $\NN^3$. Then $\mathrm{in}_<(S)$ is not Noetherian. In this case, $S$ is a regular ring since $x+y+z, xy, xy^2$ are algebraically independent over $k$.

    We show that $\mathrm{in}_<(S)$ is not Noetherian. First, for any $f(x+y+z, xy, xy^2) \in S$, $f$ is not $y^j$ after specializing $z$ to $0$, i.e., $f(x+y+0, xy, xy^2) \neq y^j$ ($j = 1, 2, \ldots$). If not, we can take $f(x+y+z, xy, xy^2) \in S$ such that $f(x+y+0, xy, xy^2) = y^j$. If we set $x = z = 0$, then $f(y,0,0) = y^j$. If we set $y = z = 0$, then $f(x,0,0) = 0$. This is a contradiction. Second, we show for all $f \in S$, $\mathrm{in}_<(f) \neq y^j \ (j=1,2,\ldots)$ (*). If there is such $f$, we can assume $f$ is homogeneous. Among the monomials of $f$, we can see $y^j > x^ay^bz^c$. In other words: 
    \begin{equation*}
        \begin{split}
            &j > a+b \\ 
            \text{or } &j = a+b \land (0, j, 0) >_{\text{lex}} (a, b, c)
        \end{split}
    \end{equation*}
    holds for all monomials $x^ay^bz^c$ of $f$ since the initial term of $f$ is $y^j$. In the latter case, we have $(a,b,c) = (0,j,0)$. In the former case, $c$ is greater than $0$ because $f$ is the degree $j$ homogeneous polynomial. Thus, evaluating at $z=0$, we obtain:
    \begin{equation*}
        f(x + y + 0, xy, xy^2) = y^j. 
    \end{equation*}
    This contradicts the first point. Assume $S$ has a finite SAGBI basis $\mathcal{F}$. Taking a large $m \gg 0$, we can find $xy^m \notin \mathrm{in}_<(\mathcal{F}) = \left\{ \mathrm{in}_<(f) \mid f \in \mathcal{F} \right\}$ (**). By induction, we have: 
    \begin{equation*}
        \mathrm{in}_<((x+y+z)xy^{m-1} - (xy)(xy^{m-2})) = \mathrm{in}_<(xy^m+xy^{m-1}z) = xy^m
    \end{equation*}
    so we have $xy^m \in \mathrm{in}_<(S)$. By the definition of a SAGBI basis, we find 
    \begin{equation*}
        xy^m = \prod_f \mathrm{in}_<(f)^{e_f}
    \end{equation*}
    for some $f \in \mathcal{F}$ and $e_f \in \NN$. However, (**) implies that we need at least two factors on the right-hand side. Since the left-hand side $xy^m$ contains exactly one $x$ and no $z$, exactly one factor in the product on the right-hand side contains $x$ (with exponent 1), and all other factors must be of the form $y^j$ for some $j > 0$. However, the existence of such a factor of the form $y^j$ contradicts (*). Hence, $\mathrm{in}_<(S)$ is not Noetherian.
\end{example}

We saw that the singularity of a commutative ring does not necessarily affect the Noetherian property of its initial algebra. However, we can find that the Poincar\'e series of the initial algebra is the same as that of the original algebra. This fact does not need the Noetherian property. 

\begin{proposition}[{\cite[Proposition 2.4]{CHV96}}]\label{prop:inipoin}
    Let $S$ be a finitely generated homogeneous subalgebra of a polynomial ring $R = k[x_1, \ldots, x_n]$ over a field $k$ and $<$ be a monomial order on $R$. Then we have $P_S(t) = P_{\mathrm{in}_<(S)}(t)$.
\end{proposition}

Using Theorem \ref{th:monomial} and Proposition \ref{prop:inipoin}, we can show that all dimensions of $S$ and $\mathrm{in}_<(S)$ are the same.

\begin{corollary}\label{cor:initialdim}
    Let $S$ be a finitely generated homogeneous subalgebra of a polynomial ring $R = k[x_1, \ldots, x_n]$ over a field $k$ and $<$ be a monomial order on $R$. Then, all dimensions of $S$ and $\mathrm{in}_<(S)$ are the same, i.e., we have:
    \begingroup
    \setlength{\arraycolsep}{0.3em}
    \begin{equation*}
        \begin{array}{*{10}{c}}
            &
            \dimgr(S)
            & = &
            \dim(S)
            & = &
            \trdeg_k(S)
            & = &
            \gkdim_k(S)
            & = &
            d(S)
            \\
            =
            &
            \dimgr(\mathrm{in}_{<}(S))
            & = &
            \dim(\mathrm{in}_{<}(S))
            & = &
            \trdeg_k(\mathrm{in}_{<}(S))
            & = &
            \gkdim_k(\mathrm{in}_{<}(S))
            & = &
            d(\mathrm{in}_{<}(S)).
        \end{array}
    \end{equation*}
    \endgroup
\end{corollary}
\begin{proof}
    The first line follows from Smoke's dimension theorem and the second from Theorem \ref{th:monomial}. The equality between the first line and the second line is given by $d(S) = d(\mathrm{in}_<(S))$ by Proposition \ref{prop:inipoin}. 
\end{proof}

By this result, we can compute the Krull dimension of a finitely generated homogeneous subalgebra $S$ of a polynomial ring $R = k[x_1, \ldots, x_n]$ over a field $k$ by passing to an initial algebra of $S$. 

We recall the Jacobian criterion for algebraic independence in modern language. 

\begin{proposition}[{\cite[Lemma 8, Theorem 22, and Remark 23]{MSS14}}]\label{prop:jacobian}
    Let $k$ be a field and $f_1, \ldots, f_m \in k[x_1, \ldots, x_n]$ be polynomials with $m \leq n$. If the Jacobian matrix
    \begin{align*}
        \mathcal{J}_{(x_1, \ldots, x_n)}(f_1, \ldots, f_m) = \left(\frac{\partial f_i}{\partial x_j}\right)
        =
        \begin{pmatrix}
            \frac{\partial f_1}{\partial x_1} & \cdots & \frac{\partial f_1}{\partial x_n} \\[1ex]
            \vdots & \ddots & \vdots \\[1ex]
            \frac{\partial f_m}{\partial x_1} & \cdots & \frac{\partial f_m}{\partial x_n}
        \end{pmatrix}
        \in \mathrm{Mat}_{m \times n}(k(x_1, \ldots, x_n))
    \end{align*}
    of $f_1, \ldots, f_m$ has rank $m$, then $f_1, \ldots, f_m$ are algebraically independent over $k$. If $\operatorname{char}(k) = 0$, the converse also holds.
\end{proposition}

This proposition is useful for computing the transcendence degree of a finitely generated subalgebra of a polynomial ring. In positive characteristic, however, algebraic independence does not necessarily imply that the Jacobian matrix has full row rank. Thus, the transcendence degree, and hence the Krull dimension, cannot always be determined by the usual Jacobian criterion. In such a situation, Corollary \ref{cor:initialdim} allows us to compute the Krull dimension by passing to the initial algebra.

\begin{example}\label{ex:initial-nonnoeth-jacobian}
    Let $k$ be a field and consider the following homogeneous subalgebra of $k[x,y]$:
    \begin{align*}
        S = k[x^8y^5+x^7y^6,\ x^{15}y^{11},\ x^{22}y^{17}].
    \end{align*}
    We take a monomial order $<$ with $x>y$. This example is obtained from \cite[Example 3.5]{HT25} by setting
    \begin{align*}
        v_1=(8,5), \qquad v_2=(7,6).
    \end{align*}
    Indeed, we have
    \begin{align*}
        v_1+v_2=(15,11), \qquad v_1+2v_2=(22,17).
    \end{align*}
    Hence, by \cite[Example 3.5]{HT25}, the initial algebra of $S$ is the monomial algebra
    \begin{align*}
        \mathrm{in}_{<}(S)
        =
        k[x^8y^5,\ x^{8+7m}y^{5+6m}\mid m\geq 1].
    \end{align*}
    In particular, $\mathrm{in}_{<}(S)$ is not Noetherian.

    Let $M$ be the submonoid of $\NN^2$ corresponding to $\mathrm{in}_{<}(S)$. Then
    \begin{align*}
        M=\langle (8,5),\ (8+7m,5+6m)\mid m\geq 1\rangle.
    \end{align*}
    Let $G$ be the group of fractions of $M$. Since $(8,5)$ and $(15,11)$ belong to $M$ and we have: 
    \begin{align*}
        \det
        \begin{pmatrix}
            8 & 15 \\
            5 & 11
        \end{pmatrix}
        =
        13
        \neq 0, 
    \end{align*}
    we have $\rank(G)=2$. Therefore, by Theorem \ref{th:gilmer}, we obtain
    \begin{align*}
        \dim(\mathrm{in}_{<}(S))=2.
    \end{align*}
    By Corollary \ref{cor:initialdim}, it follows that
    \begin{align*}
        \dim(S)=\dim(\mathrm{in}_{<}(S))=2.
    \end{align*}

    On the other hand, this dimension is not detected by the usual Jacobian computation (Proposition \ref{prop:jacobian}) in characteristic $13$. Indeed, let
    \begin{align*}
        f_1=x^8y^5+x^7y^6,\qquad
        f_2=x^{15}y^{11},\qquad
        f_3=x^{22}y^{17}.
    \end{align*}
    Then the $2\times 2$ minors of the Jacobian matrix
    \begin{align*}
        \left(\frac{\partial f_i}{\partial x_j}\right)
        =
        \begin{pmatrix}
            \frac{\partial f_1}{\partial x} & \frac{\partial f_1}{\partial y} \\[1ex]
            \frac{\partial f_2}{\partial x} & \frac{\partial f_2}{\partial y} \\[1ex]
            \frac{\partial f_3}{\partial x} & \frac{\partial f_3}{\partial y}
        \end{pmatrix}
    \end{align*}
    are
    \begin{align*}
        13x^{21}y^{15}(x-y),\qquad
        13x^{28}y^{21}(2x-y),\qquad
        13x^{36}y^{27}.
    \end{align*}
    Thus, if $\operatorname{char}(k) \neq 13$, none of these minors vanishes. In particular, each of the pairs $\{ f_1,\ f_2 \}$, $\{ f_1,\ f_3 \}$, and $\{ f_2,\ f_3 \}$ is algebraically independent over $k$. Therefore, we have $\dim(S) = \trdeg_k(S) = 2$. However, if $\operatorname{char}(k) = 13$, all these minors vanish. Therefore, in characteristic $13$, the usual Jacobian computation does not show that $\dim(S)=2$. Nevertheless, the initial algebra method above computes the dimension of $S$ by passing to the non-Noetherian monomial algebra $\mathrm{in}_{<}(S)$.

    The following Macaulay2 computation illustrates the initial terms obtained by the package \texttt{SubalgebraBases} (\cite{BCD}, \cite{BCD24}). 
    \begin{lstlisting}
        Macaulay2, version 1.25.11
        Type "help" to see useful commands
        i1 : needsPackage "SubalgebraBases"
        R = ZZ/13[x,y, MonomialOrder=>Lex];
        P = {x^8*y^5 + x^7*y^6, x^15*y^11, x^22*y^17};
        SB = sagbi(P, Limit=>100);
        isSAGBI SB
        G = flatten entries gens SB;
        inG = apply(G, f -> leadTerm f);
        inG
        o1 = SubalgebraBases
        o1 : Package
        o5 = false
               8 5   15 11   22 17   29 23   36 29   43 35   50 41
        o8 = {x y , x  y  , x  y  , x  y  , x  y  , x  y  , x  y  }
        o8 : List
    \end{lstlisting}
    The output \texttt{false} of \texttt{isSAGBI SB} is consistent with the fact that the SAGBI basis is infinite. The non-finite generation of $\mathrm{in}_{<}(S)$ is not proved by this computation, but follows from \cite[Example 3.5]{HT25}.
\end{example}

We have seen that the inequalities in our main theorem, Theorem \ref{th:main}, are actually equalities for the monomial algebra, in particular, the initial algebra. However, in general, these inequalities can be strict even if $R$ is a domain as we have already mentioned. The next section is devoted to giving some examples of this phenomenon.

\section{Examples and Counterexamples}\label{sec:examples}

In this section, we collect examples and counterexamples illustrating the scope and limitations of the results above. As we have already mentioned, the case when $R$ is Noetherian is an obvious case. Therefore, we will see some examples of non-Noetherian Hilbert--Serre rings. In particular, we will see some examples of non-Noetherian Hilbert--Serre domains $R$.

First, we give two standard examples of non-Noetherian Hilbert--Serre domains. 

\begin{example}\label{ex:1}
    Let $R = k[x, xy, xy^2, \ldots]$ be a subalgebra of $k[x, y]$ over a field $k$. Then $R$ is a non-Noetherian Hilbert--Serre domain. $R$ can be obtained as the initial algebra of $k[x+y, xy, xy^2]$ with respect to any monomial order with $x>y$. So we have that all dimensions of $R$ are the same and equal to $2$ by Corollary \ref{cor:initialdim}. Note that $\dimgr(R) = 2$ because there is a non-trivial homogeneous prime ideal $(xy, xy^2, \ldots)$.
\end{example}

\begin{example}\label{ex:2}
    Let $R = k[xy, xy^2, \ldots]$ be a subalgebra of $k[x, y]$ over a field $k$. Then $R$ is a non-Noetherian Hilbert--Serre domain. $R$ is a monoid algebra of the monoid generated by $(1,1), (1,2), \ldots$ in $\NN^2$. Hence, we have $\dimgr(R) = \dim(R) = \trdeg_k(R) = \gkdim_k(R) = d(R) = 2$ by Theorem \ref{th:monomial}. Note that $\dimgr(R) = 2$ because there is a non-trivial homogeneous prime ideal $(x-y) \cap R$.
\end{example}

There is an $\NN$-graded ring $R$ which has a given series as its Poincar\'e series but has the Krull dimension $0$. In particular, we can find a Hilbert--Serre ring $R$ such that $\dim(R) < d(R)$. 

\begin{example}\label{ex:3}
    Let $k$ be a field and $a_0 = 1, a_1, a_2, \ldots$ be any non-negative integer sequence. We consider the idealization $R = k \ltimes \bigoplus_{i=1}^\infty k^{\oplus a_i}$ and define an $\NN$-graded ring on $R$ by setting $R_n = k^{\oplus a_n}$. By the basic property of the idealization, we have $\dim(R) = 0$. However, its Poincar\'e series is given by $P_R(t) = 1 + a_1t + a_2t^2 + \cdots$. Hence, $d(R)$ can be any non-negative integer or $\infty$ depending on the choice of the sequence $a_i$. In other words, the equality in our main theorem, Theorem \ref{th:main}, does not hold in general. Note that in this example, $R$ is a domain if and only if $a_i = 0$ for all $i \geq 1$. 
\end{example}

Even if $R$ is a domain, the equality in our main theorem does not hold in general.

\begin{example}\label{ex:counter}
    Let $k$ be a field and $k(x, y)$ be a rational function field in two variables over $k$. First, let $d > 1$ be an integer. We define a graded ring $R = \bigoplus_{n=0}^\infty R_n$ as follows: 
    \begin{equation*}
        R_n = \mathrm{Span}_k\{x^ny^j \mid 0 \leq j \leq n^d\}.
    \end{equation*}
    Then $R$ is a graded subalgebra of $k(x, y)$ with $\trdeg_k(R) = 2$, and we have
    \begin{equation*}
        P_R(t) = \sum_{n=0}^{\infty} (n^d+1)t^n
    \end{equation*}
    as its Poincar\'e series. It is well-known that this is a rational function such that the order of the pole at $t=1$ is $d+1$.
\end{example}

These examples show that while Theorem \ref{th:main} provides a bound for Hilbert--Serre rings, the gap between these dimensions can be arbitrarily large in the non-Noetherian setting.

We recall that a Hilbert--Serre ring is an $\NN$-graded ring $R = \bigoplus_{n=0}^\infty R_n$ with $R_0 = k$ a field satisfying the following conditions:

\begin{itemize}
    \item[(1)] $\dim_k(R_n) < \infty$ for all $n \in \NN$,
    \item[(2)] $P_R(t) = \sum_{n=0}^\infty \dim_k(R_n)t^n$ has radius of convergence 1 or $\infty$,
    \item[(3)] $\lim_{t \to 1-0} (1-t)^d P_R(t) < \infty$ for some $d \in \NN$.
\end{itemize}

We will see three examples of non-Hilbert--Serre rings. Example \ref{ex:4} does not satisfy (1), Example \ref{ex:5} does not satisfy (2) and Example \ref{ex:6} does not satisfy (3).

\begin{example}\label{ex:4}
    Let $R = k[x_1, x_2, \ldots]$ be the polynomial ring in infinitely many variables over a field $k$. Then $R$ is an $\NN$-graded ring with $R_0 = k$. However, $\dim_k(R_n) = \infty$ for all $n > 0$, so $R$ does not satisfy condition (1).
\end{example}

\begin{example}\label{ex:5}
    Let $k$ be a field and $k(x, y)$ be a rational function field in two variables over $k$. We define a graded ring $R = \bigoplus_{n=0}^\infty R_n$ as follows:
    \begin{equation*}
        R_n = \mathrm{Span}_k\{x^ny^j \mid 0 \leq j \leq 2^n\}.
    \end{equation*}
    Then $R$ is a graded subalgebra of $k(x, y)$, and clearly $\trdeg_k(R) = 2$. However, its Poincar\'e series is given by
    \begin{equation*}
        P_R(t) = \sum_{n=0}^{\infty} (2^n+1)t^n. 
    \end{equation*}
    Its radius of convergence is $\frac{1}{2}$ since $\limsup_{n \to \infty} \sqrt[n]{2^n+1} = 2$. Hence, $R$ does not satisfy condition (2).
\end{example}

\begin{example}\label{ex:6}
    We again consider Example \ref{ex:hilbertserre} (5). Let $k[x_1, x_2^2, x_3^3, \ldots]$ be the polynomial ring in infinitely many variables over a field $k$. Its Poincar\'e series is the generating function of the partition function $p(n)$. It is known that the partition function has the following asymptotic formula:
    \begin{equation*}
        p(n) \sim \frac{1}{4n\sqrt{3}} e^{\pi \sqrt{\frac{2n}{3}}}
    \end{equation*}
    which is known as the Hardy-Ramanujan formula. Hence, its Poincar\'e series has radius of convergence $1$ since $\limsup_{n \to \infty} \sqrt[n]{p(n)} = 1$. However, $R$ has infinite Krull dimension and, by Corollary \ref{cor:finitedimension}, $R$ is not a Hilbert--Serre ring. In particular, it is impossible to find $d \in \NN$ such that $\lim_{t \to 1-0} (1-t)^d P_R(t) < \infty$. Hence, $R$ does not satisfy condition (3).
\end{example}

Finally, we see that the equality $\dimgr(R) = \dim(R)$ does not hold in general even if $R$ is a Hilbert--Serre domain. 

We can easily find a counterexample to this if we do not assume the Hilbert--Serre condition. Fix an integer $d \geq 1$ and $d'$ with $d+2 \leq d' \leq 2d+1$. First, we take a domain $A$ containing a field $k$ with $\dim(A) = d$ and $\dim(A[x]) = d'$. One can verify that such $A$ exists by \cite{Sei54}. Then we define a graded ring $R = \bigoplus_{n=0}^\infty R_n$ as follows:
\begin{equation*}
    R_0 = k,\ R_n = A x^n \text{ for } n > 0.
\end{equation*}
We show the following bijection preserving the inclusion:
\begin{equation*}
    \operatorname{Spec}(A) \to \operatorname{Proj}(R):\ \mathfrak{p} \mapsto \bigoplus_{n=1}^\infty \mathfrak{p} x^n. 
\end{equation*}
This map is well defined and injective. Conversely, for $P \in \operatorname{Proj}(R)$, the ideal $\mathfrak{p} = \{ a \in A \mid ax \in P \}$ is a prime ideal of $A$ and satisfies $P = \bigoplus_{n=1}^\infty \mathfrak{p} x^n$. Hence, the map is surjective. Therefore, we have $\dim(\operatorname{Proj}(R)) = \dim(A) = d$. Since $\mathfrak{m} = \bigoplus_{n=1}^\infty A x^n$ is the maximal homogeneous ideal of $R$, appending $\mathfrak{m}$ to a maximal chain in $\operatorname{Proj}(R)$ gives $\dimgr(R) = d+1$. We also have the following inclusion-preserving bijection:
\begin{equation*}
    \operatorname{Spec}(R) \setminus V_R(\mathfrak{m}) \to \operatorname{Spec}(A[x]) \setminus V_{A[x]}(\mathfrak{m}):\ \mathfrak{p} \mapsto (\mathfrak{p} \underset{A[x]}{:} \mathfrak{m}),
\end{equation*}
where $\mathfrak{m} = \bigoplus_{n=1}^\infty A x^n$ is an ideal of both $R$ and $A[x]$, and $(\mathfrak{p} \underset{A[x]}{:} \mathfrak{m}) = \left\{ f \in A[x] \mid f\mathfrak{m} \subseteq \mathfrak{p} \right\}$. This is the inverse of the natural contraction. Moreover, $\mathfrak{m}$ is a height one prime ideal of $A[x]$; if not, there would exist a prime ideal $\mathfrak{q}$ with $0 \subsetneq \mathfrak{q} \subsetneq \mathfrak{m}$, and choosing a non-zero element of $\mathfrak{q}$ of minimal $x$-degree gives a contradiction. Since $\dim(A[x]) = d'$, this yields $d' \leq \dim(R)$. Therefore, $d+1 = \dimgr(R) < d' \leq \dim(R)$.

Now, we give a counterexample to the equality $\dimgr(R) = \dim(R)$ under the Hilbert--Serre condition. The essential idea of the construction was given by Samuel Griffiths. 

\begin{example}\label{ex:countergr}
    We construct some Hilbert--Serre domains using divisors on a projective scheme. On divisors, \cite[Chapter II]{Har77} is a standard reference.

    Let $k$ be a field which has countable cardinality (for instance, $k = \overline{\QQ}$, $\overline{\FF_p}$ and so on. ), $r \in \NN_{>0}$ an integer, $X = \PP_k^r = \operatorname{Proj}(k[x_0, \ldots, x_r])$ a projective scheme over $k$, and $K = k(x_1, \ldots, x_r)$ the rational function field of $X$. Since $k$ is countable, we can enumerate all prime divisors of $X$ as $C_0, C_1, C_2, \ldots$. We assume $C_0 = \{x_0 = 0\}$ and $C_1 = \{x_1 = 0\}$ hyperplane divisors of $X$. We define a family of effective Weil divisors $\{ E_n \}_{n=0}^\infty$ on $X$ as follows:
    \begin{equation*}
        E_0 = 0,\ E_n = n^c C_0 + \sum_{i=1}^\infty \left\lfloor \frac{n}{\omega_i} \right\rfloor C_i \text{ for } n > 0,
    \end{equation*}
    where $c \in \NN_{>0}$ is an arbitrary fixed integer and $\omega_i = 2^i \deg(C_i)$. Note that $\{E_n\}$ satisfies the following conditions:
    \begin{itemize}
        \item[(1)] $E_n$ is an effective Weil divisor for all $n \in \NN$,
        \item[(2)] $E_n + E_m \leq E_{n+m}$ for all $n, m \in \NN$.
    \end{itemize}
    For a Weil divisor $D$ on $X$, we define $L(D)$ as follows:
    \begin{equation*}
        L(D) = \{ f \in K^\times \mid \div(f) + D \geq 0 \} \cup \{0\} = H^0(X, \mathcal{O}_X(D)).
    \end{equation*}
    Note that every Weil divisor on $X$ is linearly equivalent to $\deg(D)C_0$; see \cite[Proposition II.6.4 (a)]{Har77}. Let $D_0 = 0$ and $D_n = E_{n-1}$ for $n \geq 1$. We define a graded ring $R = \bigoplus_{n=0}^\infty R_n$ as a subalgebra of $K[T]$ as follows:
    \begin{equation*}
        R_n = L(D_n) T^n.
    \end{equation*}
    In other words, $R$ is displayed as follows:
    \begin{equation*}
        R = \bigoplus_{n=0}^\infty L(D_n) T^n = \bigoplus_{n=0}^\infty H^0(X, \mathcal{O}_X(D_n)) T^n \subset K[T].
    \end{equation*}
    First, we show that $d(R) = cr + 1$. By the asymptotic Riemann--Roch theorem (\cite[0BJ8]{Spaut26}), we have:
    \begin{align*}
        \dim_k(R_n) &= \dim_k(H^0(X, \mathcal{O}_X(D_n))) \\
        &= \dim_k(H^0(X, \mathcal{O}_X(C_0)^{\otimes \deg(D_n)})) \\
        &= A \deg(D_n)^r + O(\deg(D_n)^{r-1}),
    \end{align*}
    where $A$ is a positive constant independent of $n$. Moreover, for $n \in \NN$, we have:
    \begin{equation*}
        n^c \leq \deg(E_n) = n^c + \sum_{i=1}^\infty \left\lfloor \frac{n}{\omega_i} \right\rfloor \deg(C_i) \leq n^c + n \sum_{i=1}^{\infty} \frac{1}{2^i} = n^c + n.
    \end{equation*}    
    Hence, $\dim_k(R_n)$ has growth of order $n^{cr}$, and thus $d(R) = cr + 1$.\\     
    Second, we show $\trdeg_k(R) = r + 1$. We can show the following equality: 
    \begin{equation*}
        \bigcup_{n=0}^\infty L(D_n) = K.
    \end{equation*}
    $\subset$ is clear. We show $\supset$. For $f\in K^\times$, write
    \begin{equation*}
        \div(f)=\sum_{\alpha=1}^{p}a_\alpha C_{s_\alpha}
    -\sum_{\beta=1}^{q}b_\beta C_{t_\beta}.
    \end{equation*}
    Only finitely many prime divisors occur in the pole part of $\div(f)$. If $t_\beta=0$, then the coefficient $n^c$ of $C_0$ in $E_n$ is eventually at least $b_\beta$. If $t_\beta\ge1$, then for $n\gg0$ we have:
    \begin{equation*}
        \left\lfloor\frac n{\omega_{t_\beta}}\right\rfloor\ge b_\beta.
    \end{equation*}
    Hence, we have $\div(f)+E_n\ge0$ for $n\gg0$. Thus, we have $f \in L(D_{n+1})$ for sufficiently large $n$. Therefore, we have $\bigcup_{n=0}^\infty L(D_n) = K$. We show $\Frac(R) = K(T)$. $\subset$ is clear. We show $\supset$. Since $1 \in L(D_1)$, we have $T \in R$. Hence, for $f\in K$, there exists $N$ such that $f\in L(D_N)$. Then $fT^N\in R_N$, and hence
    \begin{equation*}
        f=\frac{fT^N}{T^N}\in \Frac(R).
    \end{equation*}
    Thus, we find that $\Frac(R)$ is a field which contains $K$ and $T$. Hence, we have $K(T) \subset \Frac(R)$. Therefore, we have $\trdeg_k(R) = r + 1$. \\ 
    Third, we show $\dim(R) = 2$. We can show $R_{T} = K[T, T^{-1}]$. $\subset$ is clear. We show $\supset$. For $f \in K^\times$, we can find $N$ such that $f \in L(D_N)$. Then we have $fT^N \in R_N \subset R$ and thus $f = \frac{fT^N}{T^N} \in R_T$. Hence, we have $K \subset R_T$. Since $T \in R$, we have $K[T, T^{-1}] \subset R_T$. Therefore, we have $R_T = K[T, T^{-1}]$. Hence, the length of a chain of prime ideals of $R$ which do not contain $T$ is at most $1$. We show that if a prime ideal of $R$ contains $T$, then this prime ideal is $R_+ = \bigoplus_{n>0} R_n$. Let $m, n >0$ be integers and $f \in L(D_m)$ and $g \in L(D_n)$. Then we have:
    \begin{align*}
        \div(fg) + D_{m+n-1} &= \div(f) + \div(g) + D_{m+n-1} \\
        &\geq \div(f) + \div(g) + E_{m+n-2}\\
        &\geq \div(f) + E_{m-1} + \div(g) + E_{n-1} \\
        &\geq \div(f) + D_{m} + \div(g) + D_{n} \\
        &\geq 0.
    \end{align*}
    Hence, we have $fg \in L(D_{m+n-1})$. Thus,
    \begin{equation*}
        fT^m \cdot gT^n = fg T^{m+n-1} \cdot T \in R_{m+n-1} \cdot TR \subset TR. 
    \end{equation*}
    Therefore, $(R_+)^2 \subset TR$, i.e., $R_+ = \sqrt{TR}$. At this point, we have $\dim(R) \leq 2$. We show that $R$ has a non-trivial prime ideal. Let $v_{C_1}$ be a discrete valuation of $K$ associated to $C_1$. We find that the kernel of the following ring homomorphism is a non-trivial prime ideal of $R$:
    \begin{equation*}
        \varphi \colon R \to K \colon \sum f_n T^n \mapsto \sum f_n x_1^n.
    \end{equation*}
    Let $\mathfrak{p} = \ker(\varphi)$. We show $\mathfrak{p} \subsetneq R_+$. Let $f = f_0 + \sum_{n>0} f_n T^n \in \mathfrak{p}$. Since $\div(f_n) + D_n \geq 0$ for $n>0$, we have $v_{C_1}(f_n) \geq -\left\lfloor \frac{n-1}{\omega_1} \right\rfloor$ for $n>0$. Hence, we have $v_{C_1}(f_nx_1^n) = n + v_{C_1}(f_n) \geq n - \left\lfloor \frac{n-1}{\omega_1} \right\rfloor \geq 1$. Thus, specializing $x_1$ to $0$ in $\varphi(f) = \sum f_n x_1^n$ gives $f_0 = 0$. Hence, we have $f \in R_+$. Therefore, we have $\mathfrak{p} \subset R_+$. Since $T \in R_+ \setminus \mathfrak{p}$, we have $\mathfrak{p} \subsetneq R_+$. We can also show $\mathfrak{p} \neq 0$. For $x_1$, there is an integer $N$ such that $x_1 \in L(D_N)$. Then we have $x_1 T^N \in R_N \subset R$. Thus, we can easily check $0 \neq x_1 T^N - T^{N+1} \in \mathfrak{p}$. Hence, we have $0 \subsetneq \mathfrak{p} \subsetneq R_+$. Therefore, we have $\dim(R) = 2$. 

    Finally, we show $\dimgr(R) = 1$. Let $\mathfrak{q}$ be a homogeneous prime ideal of $R$. If $T \in \mathfrak{q}$, then $\mathfrak{q} = R_+$ since $R_+ = \sqrt{TR}$. If not, then $\mathfrak{q}R_T$ is a homogeneous prime ideal of $R_T = K[T, T^{-1}]$. However, there is only one homogeneous prime ideal of $K[T, T^{-1}]$, which is $0$. Hence, we have $\mathfrak{q}R_T = 0$ and thus $\mathfrak{q} = 0$. Therefore, we have $\dimgr(R) = 1$.

    In summary, we have $\dimgr(R) = 1 \leq \dim(R) = 2 \leq \trdeg_k(R) = r + 1 \leq d(R) = cr + 1$. In particular, if we take $c = 2$ and $r = 2$, then we have $\dimgr(R) = 1 < \dim(R) = 2 < \trdeg_k(R) = \gkdim_k(R) = 3 < d(R) = 5$. This shows that the inequalities in Theorem \ref{th:main} can be strict even if $R$ is a Hilbert--Serre domain.
\end{example}

Finally, we leave the following question open.

\begin{question}\label{qu:open}
    Let $a,b,c,d,e \in \NN$ be integers. What conditions on $a,b,c,d,e$ are necessary and sufficient for the existence of a Hilbert--Serre domain $R$ such that $\dimgr(R) = a,\ \dim(R) = b,\ \trdeg_k(R) = c,\ \gkdim_k(R) = d,\ d(R) = e$ ?
\end{question}

As necessary conditions, we have $a \leq b \leq c = d \leq e$ by Theorem \ref{th:main}. Moreover, $\lfloor \frac{b}{2} \rfloor \leq a$ by Proposition \ref{prop:dimupper}. 

\bibliographystyle{amsalpha}
\bibliography{references}

\end{document}